\documentclass[11pt]{amsart}
\usepackage{amsmath,amssymb,,latexsym,esint,cite}
\usepackage{verbatim}
\usepackage[left=3.1cm,right=3.1cm,top=3.2cm,bottom=3.2cm]{geometry}

\usepackage{color,enumitem,graphicx}
\usepackage[colorlinks=true,urlcolor=blue, citecolor=red,linkcolor=blue,linktocpage,pdfpagelabels, bookmarksnumbered,bookmarksopen]{hyperref}

\usepackage[hyperpageref]{backref}
\usepackage[english]{babel}

\newcommand{\abs}[1]{\left|#1\right|}
\newcommand{\bdry}[1]{\partial #1}
\newcommand{\bgset}[1]{\big\{#1\big\}}
\newcommand{\A}{{\mathcal A}}

\newcommand{\F}{{\mathcal F}}

\newcommand{\dist}[2]{\text{dist}\, (#1,#2)}
\newcommand{\eps}{\varepsilon}
\newcommand{\half}{\frac{1}{2}}
\newcommand{\id}[1][]{id_{\, #1}}
\newcommand{\incl}{\subset}
\newcommand{\M}{{\mathcal M}}
\newcommand{\N}{\mathbb N}
\newcommand{\norm}[2][]{\left\|#2\right\|_{#1}}
\renewcommand{\O}{\text{O}}
\newcommand{\PS}[1]{$(\text{PS})_{#1}$}
\newcommand{\pnorm}[2][]{\if #1'' \left|#2\right|_p \else \left|#2\right|_{#1} \fi}
\newcommand{\R}{\mathbb R}
\newcommand{\RP}{\R \text{P}}
\newcommand{\restr}[2]{\left.#1\right|_{#2}}
\newcommand{\seq}[1]{\left(#1\right)}
\newcommand{\set}[1]{\left\{#1\right\}}
\newcommand{\strictsubset}{\subset \subset}
\newcommand{\Z}{\mathbb Z}

\newcommand{\ds}[1]{\displaystyle #1}
\DeclareMathOperator{\supp}{supp}

\newenvironment{enumroman}{\begin{enumerate}
\renewcommand{\theenumi}{$(\roman{enumi})$}
\renewcommand{\labelenumi}{$(\roman{enumi})$}}{\end{enumerate}}

\newenvironment{properties}[1]{\begin{enumerate}
		\renewcommand{\theenumi}{$(#1_\arabic{enumi})$}
		\renewcommand{\labelenumi}{$(#1_\arabic{enumi})$}}{\end{enumerate}}

\newtheorem{lemma}{Lemma}[section]
\newtheorem{proposition}[lemma]{Proposition}
\newtheorem{theorem}[lemma]{Theorem}
\newtheorem{corollary}[lemma]{Corollary}
\theoremstyle{definition}

\theoremstyle{remark}
\newtheorem{remark}[lemma]{Remark}

\numberwithin{equation}{section}

\title[Fractional $p$-Laplacian problems involving critical Hardy-Sobolev exponents]{The Brezis-Nirenberg problem for the fractional $p$-Laplacian involving critical Hardy-Sobolev exponents }

\author[Y.\ Yang]{Yang Yang}

\address[Y. Yang]{School of Science
	\newline\indent
	Jiangnan University
	\newline\indent
	Wuxi, Jiangsu 214122, China}
\email{yynjnu@126.com}

\thanks{The author thanks Prof. Kanishka Perera and Prof. Marco Squassina for their help and many valuable discussions. Project supported by NSFC(No. 11501252, No. 11571176).}

\begin{document}

\begin{abstract}
We obtain existence, multiplicity, and bifurcation results for the Brezis-Nirenberg problem for the fractional $p$\nobreakdash-Laplacian operator, involving critical Hardy-Sobolev exponents.  Our results are mainly extend results in the literature for $\alpha=0$. In the absence of an explicit formula for a minimizer in the fractional Hardy-Sobolev inequality  $\alpha\not= 0$, we get around this difficulty by working with certain asymptotic estimates for minimizers recently obtained in \cite{MM}.
\end{abstract}

\maketitle

\begin{center}
	\begin{minipage}{12cm}
		\small
		\tableofcontents
	\end{minipage}
\end{center}

\medskip

\section{Introduction and main results}

For $1 < p < \infty$, $s \in (0,1)$, and $N > sp$, the fractional $p$-Laplacian $(- \Delta)_p^s$ is the nonlinear nonlocal operator defined on smooth functions by
\[
(- \Delta)_p^s\, u(x) = 2\, \lim_{\eps \searrow 0} \int_{B_\eps(x)^c} \frac{|u(x) - u(y)|^{p-2}\, (u(x) - u(y))}{|x - y|^{N+sp}}\, dy, \quad x \in \R^N.
\]
This definition is consistent, up to a normalization constant depending on $N$ and $s$, with the usual definition of the linear fractional Laplacian operator $(- \Delta)^s$ when $p = 2$. There is, currently, a rapidly growing literature on problems involving these nonlocal operators. In particular, fractional $p$-eigenvalue problems have been studied in Brasco et al.\! \cite{BrPaSq}, Brasco and Parini \cite{BrPa},
Franzina and Palatucci \cite{FrPa}, Iannizzotto and Squassina \cite{MR3245079}, and Lindgren and Lindqvist \cite{MR3148135}.
Regularity of solutions was obtained in Brasco and Lindgren \cite{BraLin}, Di Castro et al.\! \cite{DiKuPa,MR3237774}, Iannizzotto et al.\! \cite{IaMoSq}, Kuusi et al.\! \cite{MR3339179}, and Lindgren \cite{Li}. Existence via Morse theory was investigated in Iannizzotto et al.\! \cite{IaLiPeSq}. Critical case was considered in Mosconi and Squassina\cite{MS1,MS2},  Mosconi et al.\! \cite{MPSY} and Xiang et al.\!\cite{XZZ}.
This operator appears in some recent works, see \cite{AMRT,IN} as well as \cite{Ca} for the motivations, that led to its introduction.

Let $\Omega$ be a bounded domain in $\R^N$ with Lipschitz boundary. We consider the problem
\begin{equation} \label{1}
\begin{cases}
(- \Delta)_p^s\, u  = \lambda\, |u|^{p-2}\, u +\frac{|u|^{{{p_s^\ast(\alpha)} - 2}}}{|x|^\alpha}\, u & \text{in $\Omega$}  \\
u  = 0  & \text{in $\R^N \setminus \Omega$},
\end{cases}
\end{equation}
where $\lambda > 0$, $0<\alpha<sp<N$ and ${p_s^\ast(\alpha)} = p(N-\alpha)/(N - sp)$ is the fractional critical Hardy-Sobolev exponent. Let us recall the weak formulation of problem \eqref{1}. Let

\[
[u]_{s,p} = \left(\int_{\R^{2N}} \frac{|u(x) - u(y)|^p}{|x - y|^{N+sp}}\, dx dy\right)^{1/p}
\]
be the Gagliardo seminorm of a measurable function $u : \R^N \to \R$, and let
\[
W^{s,p}(\R^N) = \set{u \in L^p(\R^N) : [u]_{s,p} < \infty}
\]
be the fractional Sobolev space endowed with the norm
\[
\norm[s,p]{u} = \big(\pnorm{u}^p + [u]_{s,p}^p\big)^{1/p},
\]
where $\pnorm{\cdot}$ is the norm in $L^p(\R^N)$. We work in the closed linear subspace
\[
W^{s,p}_0(\Omega) = \set{u \in W^{s,p}(\R^N) : u = 0 \text{ a.e.\! in } \R^N \setminus \Omega},
\]
equivalently renormed by setting $\norm{\cdot} = [\cdot]_{s,p}$, which is a uniformly convex Banach space.
%By \cite[Theorems 6.5 \& 7.1]{MR2944369},
The imbedding $W^{s,p}_0(\Omega) \hookrightarrow L^r(\Omega)$ is continuous for $r \in [1,{p_s^\ast}]$ and compact for $r \in [1,{p_s^\ast})$. Let

\[\pnorm[{p_s^\ast(\alpha)}]{u}=\left(\int_{\R^N}\frac{|u|^{{{p_s^\ast(\alpha)}}}}{|x|^\alpha}dx\right)^{1/{{p_s^\ast(\alpha)}}}.\]
 A function $u \in W^{s,p}_0(\Omega)$ is a weak solution of problem \eqref{1} if

\begin{align*}
\int_{\R^{2N}} \frac{|u(x) - u(y)|^{p-2}\, (u(x) - u(y))\, (v(x) - v(y))}{|x - y|^{N+sp}}\, dx dy &= \lambda \int_\Omega |u|^{p-2}\, uv\, dx \\
&+ \int_\Omega \frac{|u|^{{{p_s^\ast(\alpha)} - 2}}}{|x|^\alpha}\, uv\, dx, \quad \forall v \in W^{s,p}_0(\Omega).
\end{align*}

If $\alpha=0$, problem \eqref{1} reduces to the critical fractional p-Laplacian problem
\begin{equation} \label{1.2}
\begin{cases}
(- \Delta)_p^s\, u = \lambda |u|^{p-2} u + |u|^{p_s^\ast - 2}\, u & \text{in $\Omega$}  \\
u= 0 & \text{in $\R^N \setminus \Omega$},
\end{cases}
\end{equation}
where $\lambda > 0$ and $p_s^\ast = Np/(N - sp)$. This nonlocal problem generalizes the well-known Brezis-Nirenberg problem, which has been extensively studied beginning with the seminal paper \cite{MR709644} (see, e.g., \cite{MR779872,MR831041,MR829403,MR1009077,MR987374,MR1124117,MR1083144,MR1154480,MR1306583,MR1473856,MR1441856,MR1491613,MR1695021,MR1784441,MR1961520, MR3089742,MR3237009,MR3060890,MR3271254} and references therein). Consequently, many results known in the local case $s = 1$ have been extended to problem \eqref{1.2}. In particular, S. Mosconi, K. Perera, M. Squassina,
and Y. Yang \cite{MPSY} have shown that problem \eqref{1.2} has a nontrivial weak solution in the following cases:
\begin{enumroman}
\item $N = sp^2$ and $\lambda < \lambda_1$;
\item $N > sp^2$ and $\lambda$ is not one of the eigenvalues $\lambda_k$;
\item $N^2/(N + s) > sp^2$;
\item $(N^3 + s^3 p^3)/N\, (N + s) > sp^2$ and $\bdry{\Omega} \in C^{1,1}$.
\end{enumroman}

This extends to the fractional setting some well-known results of Brezis and Nirenberg \cite{MR709644}, Capozzi et al.\! \cite{MR831041}, Zhang \cite{MR987374}, and Gazzola and Ruf \cite{MR1441856} for critical Laplacian problems.

In the present paper we consider the case $\alpha \ne 0$ of problem \eqref{1}.  This presents us with two serious new difficulties. Let
\begin{equation} \label{3}
S= \inf_{u \in {W}^{s,p}(\R^N) \setminus \set{0}}\, \frac{\norm{u}^p}{\pnorm[{p_s^\ast(\alpha)}]{u}^p},
\end{equation}
which is positive by the fractional Hardy-Sobolev inequality. Our  major difficulty is the lack of an explicit formula for a minimizer for $S$.  We will get around this difficulty by working with certain asymptotic estimates for minimizers recently obtained in Marano et al.\! \cite{MM}.

Our second main difficulty is   that the linking arguments based on eigenspaces of $(- \Delta)^s$ used in the case $p = 2$ do {\em not} work when $p \ne 2$ since the nonlinear operator $(- \Delta)_p^s$ does not have linear eigenspaces. We will use a more general construction based on sublevel sets as in Perera and Szulkin \cite{MR2153141} (see also Perera et al.\! \cite[Proposition 3.23]{MR2640827}). Moreover, the standard sequence of variational eigenvalues of $(- \Delta)_p^s$ based on the genus does not give enough information about the structure of the sublevel sets to carry out this linking construction. Therefore we will use a different sequence of eigenvalues introduced in Iannizzotto et al.\! \cite{IaLiPeSq} that is based on the $\Z_2$-cohomological index of Fadell and Rabinowitz \cite{MR57:17677}.

The Dirichlet spectrum of $(- \Delta)_p^s$ in $\Omega$ consists of those $\lambda \in \R$ for which the problem
\begin{equation} \label{1.4}
\begin{cases}
(- \Delta)_p^s\, u  = \lambda\, |u|^{p-2}\, u & \text{in $\Omega$}  \\
u  = 0 & \text{in $\R^N \setminus \Omega$}
\end{cases}
\end{equation}
has a nontrivial weak solution. Although a complete description of the spectrum is not known when $p \ne 2$, we can define an increasing and unbounded sequence of variational eigenvalues via a suitable minimax scheme. The standard scheme based on the genus does not give the index information necessary for our purposes here, so we will use the following scheme based on the cohomological index as in Iannizzotto et al.\! \cite{IaLiPeSq} (see also Perera \cite{MR1998432}). Let
\[
\Psi(u) = \frac{1}{\pnorm{u}^p}, \quad u \in \M = \bgset{u \in W^{s,p}_0(\Omega) : \norm{u} = 1}.
\]
Then eigenvalues of problem \eqref{1.4} coincide with critical values of $\Psi$. We use the standard notation
\[
\Psi^a = \set{u \in \M : \Psi(u) \le a}, \quad \Psi_a = \set{u \in \M : \Psi(u) \ge a}, \quad a \in \R
\]
for the sublevel sets and superlevel sets, respectively. Let $\F$ denote the class of symmetric subsets of $\M$, and set
\[
\lambda_k := \inf_{M \in \F,\; i(M) \ge k}\, \sup_{u \in M}\, \Psi(u), \quad k \in \N.
\]
Then $0 < \lambda_1 < \lambda_2 \le \lambda_3 \le \cdots \to + \infty$ is a sequence of eigenvalues of problem \eqref{1.4}, and
\begin{equation} \label{5}
\lambda_k < \lambda_{k+1} \implies i(\Psi^{\lambda_k}) = i(\M \setminus \Psi_{\lambda_{k+1}}) = k
\end{equation}
(see Iannizzotto et al.\! \cite[Proposition 2.4]{IaLiPeSq}). The asymptotic behavior of these eigenvalues was recently studied in Iannizzotto and Squassina \cite{MR3245079}. Making essential use of the index information in \eqref{5}, we will prove the following theorem.

\begin{theorem}\label{Theorem 1}
Let $1 < p < \infty$, $s \in (0,1)$, $0<\alpha<sp<N$, and $\lambda > 0$. Then problem \eqref{1} has a nontrivial weak solution in the following cases:
\begin{enumroman}
\item $N = sp^2$ and $\lambda < \lambda_1$;
\item $N > sp^2$ and $\lambda$ is not one of the eigenvalues $\lambda_k$;
\item $[(N-\alpha)N+\alpha\,s\,p(1+p)]/(N + s) > sp^2$,
\item $ [N^2(N-\alpha) + s^3 p^3 +\alpha\,s\,p\,(N-sp)]/N\, (N -\alpha + s) > sp^2,$ and $\bdry{\Omega} \in C^{1,1}$.
\end{enumroman}
\end{theorem}

We also prove the following bifurcation and multiplicity results for problem \eqref{1} that do not require $N\geq sp^2$.

Set \[V_\alpha(\Omega)=\int_\Omega|x|^{\frac{\alpha(N-sp)}{sp-\alpha}}dx,\]

and note that
\begin{equation}\label{V}
\int_\Omega|u|^pdx\leq V_\alpha(\Omega)^{\frac{sp-\alpha}{N-\alpha}}
\left(\int_\Omega\frac{|u|^{p^\ast_s(\alpha)}}{|x|^\alpha}dx\right)^{p/p^\ast_s(\alpha)}, \forall u\in W_0^{s,p}(\Omega),
\end{equation}
by the H\"{o}lder inequality.

\begin{theorem} \label{Theorem 2}
\noindent
\begin{enumroman}
\item \label{Theorem 1.i} If
\[
\lambda_1 - \frac{S}{V_\alpha(\Omega)^{(sp-\alpha)/(N-\alpha)}}<\lambda < \lambda_1,
\]
then problem \eqref{1} has a pair of nontrivial solutions $\pm\, u^\lambda$ such that $u^\lambda \to 0$ as $\lambda \nearrow \lambda_1$.

\item \label{Theorem 1.ii} If $\lambda_k \le \lambda < \lambda_{k+1} = \cdots = \lambda_{k+m} < \lambda_{k+m+1}$ for some $k, m \in \N$ and
\begin{equation} \label{555}
\lambda > \lambda_{k+1} - \frac{S}{V_\alpha(\Omega)^{(sp-\alpha)/(N-\alpha)}},
\end{equation}
then problem \eqref{1} has $m$ distinct pairs of nontrivial solutions $\pm\, u^\lambda_j,\, j = 1,\dots,m$ such that $u^\lambda_j \to 0$ as $\lambda \nearrow \lambda_{k+1}$.
\end{enumroman}
\end{theorem}

In particular, we have the following existence result.

\begin{corollary}
Problem \eqref{1} has a nontrivial solution for all $\lambda \in \ds{\bigcup_{k=1}^\infty} \left(\lambda_k - \frac{S}{V_\alpha(\Omega)^{(sp-\alpha)/(N-\alpha)}}\right).$
\end{corollary}

We note that $\lambda_1 \ge \frac{S}{V_\alpha(\Omega)^{(sp-\alpha)/(N-\alpha)}}$. Indeed, if $\varphi_1$ is an eigenfunction associated with $\lambda_1$,
\[
\lambda_1 = \frac{\norm{\varphi_1}^p}{\pnorm[p]{\varphi_1}^p} \ge \frac{S \pnorm[p^\ast_s(\alpha)]{\varphi_1}^{p}}{{\pnorm[p]{\varphi_1}^p}} \ge \frac{S}{V_\alpha(\Omega)^{(sp-\alpha)/(N-\alpha)}}
\]
by the H\"{o}lder inequality.

\noindent
These theorems extends to the fractional setting some well-known results of Garc{\'{\i}}a Azorero and Peral Alonso \cite{MR912211}, Egnell \cite{MR956567}, Guedda and V{\'e}ron \cite{MR1009077}, Arioli and Gazzola \cite{MR1741848}, and Degiovanni and Lancelotti \cite{MR2514055} for critical $p$-Laplacian problems,
and Perera and Zou \cite{PZ} for $p$-Laplacian problems involving critical Hardy-Sobolev exponents.

%\begin{notations}
%We use the following notations throughout the paper. For $a \in \R$ and $q > 0$, we write $a^q = |a|^{q-1}\, a$. For $1 \le q \le \infty$, $\pnorm[q]{\cdot}$ denotes the norm in $L^q(\Omega)$ and
%\[
%q' = \begin{cases}
%\infty, & q = 1\\[5pt]
%q/(q - 1), & 1 < q < \infty\\[5pt]
%1, & q = \infty
%\end{cases}
%\]
%is the H\"{o}lder conjugate of $q$.
%\end{notations}

\section{Preliminaries}

\subsection{Cohomological index}

Let us recall the definition of the cohomological index. Let $W$ be a Banach space and let $\A$ denote the class of symmetric subsets of $W \setminus \set{0}$. For $A \in \A$, let $\overline{A} = A/\Z_2$ be the quotient space of $A$ with each $u$ and $-u$ identified, let $f : \overline{A} \to \RP^\infty$ be the classifying map of $\overline{A}$, and let $f^\ast : H^\ast(\RP^\infty) \to H^\ast(\overline{A})$ be the induced homomorphism of the Alexander-Spanier cohomology rings. The cohomological index of $A$ is defined by
\[
i(A) = \begin{cases}
\sup \set{m \ge 1 : f^\ast(\omega^{m-1}) \ne 0}, & A \ne \emptyset\\[5pt]
0, & A = \emptyset,
\end{cases}
\]
where $\omega \in H^1(\RP^\infty)$ is the generator of the polynomial ring $H^\ast(\RP^\infty) = \Z_2[\omega]$. For example, the classifying map of the unit sphere $S^{m-1}$ in $\R^m,\, m \ge 1$ is the inclusion $\RP^{m-1} \incl \RP^\infty$, which induces isomorphisms on $H^q$ for $q \le m - 1$, so $i(S^{m-1}) = m$.

The following proposition summarizes the basic properties of the cohomological index.

\begin{proposition}[Fadell-Rabinowitz {\cite[Theorem 5.1]{MR57:17677}}] \label{Proposition 2}
	The index $i : \A \to \N \cup \set{0,\infty}$ has the following properties:
	\begin{properties}{i}
		\item Definiteness: $i(A) = 0$ if and only if $A = \emptyset$;
		\item \label{i2} Monotonicity: If there is an odd continuous map from $A$ to $B$ (in particular, if $A \subset B$), then $i(A) \le i(B)$. Thus, equality holds when the map is an odd homeomorphism;
		\item Dimension: $i(A) \le \dim W$;
		\item Continuity: If $A$ is closed, then there is a closed neighborhood $N \in \A$ of $A$ such that $i(N) = i(A)$. When $A$ is compact, $N$ may be chosen to be a $\delta$-neighborhood $N_\delta(A) = \set{u \in W : \dist{u}{A} \le \delta}$;
		\item Subadditivity: If $A$ and $B$ are closed, then $i(A \cup B) \le i(A) + i(B)$;
		\item \label{i6} Stability: If $SA$ is the suspension of $A \ne \emptyset$, obtained as the quotient space of $A \times [-1,1]$ with $A \times \set{1}$ and $A \times \set{-1}$ collapsed to different points, then $i(SA) = i(A) + 1$;
		\item \label{i7} Piercing property: If $A$, $A_0$ and $A_1$ are closed, and $\varphi : A \times [0,1] \to A_0 \cup A_1$ is a continuous map such that $\varphi(-u,t) = - \varphi(u,t)$ for all $(u,t) \in A \times [0,1]$, $\varphi(A \times [0,1])$ is closed, $\varphi(A \times \set{0}) \subset A_0$ and $\varphi(A \times \set{1}) \subset A_1$, then $i(\varphi(A \times [0,1]) \cap A_0 \cap A_1) \ge i(A)$;
		\item Neighborhood of zero: If $U$ is a bounded closed symmetric neighborhood of $0$, then $i(\bdry{U}) = \dim W$.
	\end{properties}
\end{proposition}

\subsection{Abstract critical point theorems}
 We will proof  Theorem \ref{Theorem 1} and Theorem \ref{Theorem 2} using the following abstract critical point theorems
proved in Yang and Perera (cf.\ \cite[Theorem 2.2]{YaPe2}) and in Perera, Squassina, and Yang(cf.\ \cite[Theorem 2.2]{PeSqYa2})  respectively.

Recall that $I$ satisfies the Palais-Smale compactness condition at the level $c\in\R$ or the $(PS)_c$ condition for short, if every
sequence $\{u_j\}\subset W$ such that $I(u_j)\to c$ and $I'(u_j)\to0$ has a convergent subsequence.
\begin{theorem}
	\label{Theorem 2.1}
Let $W$ be a Banach space, let $S = \set{u \in W : \norm{u} = 1}$ be the unit sphere in $W$, and let $\pi : W \setminus \set{0} \to S,\, u \mapsto u/\norm{u}$ be the radial projection onto $S$. Let $I$ be a $C^1$-functional on $W$ and let $A_0$ and $B_0$ be disjoint nonempty closed symmetric subsets of $S$ such that
\[
i(A_0) = i(S \setminus B_0) < \infty.
\]
Assume that there exist $R > r > 0$ and $v \in S \setminus A_0$ such that
\[
\sup I(A) \le \inf I(B), \qquad \sup I(X) < \infty,
\]
where
\begin{align*}
A &= \set{tu : u \in A_0,\, 0 \le t \le R} \cup \set{R\, \pi((1 - t)\, u + tv) : u \in A_0,\, 0 \le t \le 1},\\
B &= \set{ru : u \in B_0},   \\
X &= \set{tu : u \in A,\, \norm{u} = R,\, 0 \le t \le 1}.
\end{align*}
Let $\Gamma = \set{\gamma \in C(X,W) : \gamma(X) \text{ is closed and} \restr{\gamma}{A} = \id[A]}$, and set
\[
c := \inf_{\gamma \in \Gamma}\, \sup_{u \in \gamma(X)}\, I(u).
\]
Then
\begin{equation} \label{43}
\inf I(B) \le c \le \sup I(X),
\end{equation}
in particular, $c$ is finite. If, in addition, $I$ satisfies the {\em \PS{c}} condition, then $c$ is a critical value of $I$.
\end{theorem}

\noindent
Theorem \ref{Theorem 2} generalizes the linking theorem of Rabinowitz \cite{MR0488128}. The linking construction in its proof was also used in Perera and Szulkin \cite{MR2153141} to obtain nontrivial solutions of $p$-Laplacian problems with nonlinearities that interact with the spectrum. A similar construction based on the notion of cohomological linking was given in Degiovanni and Lancelotti \cite{MR2371112}. See also Perera et al.\! \cite[Proposition 3.23]{MR2640827}.\\

Now let $I$ be an even $C^1$-functional on $W$ and let $\A^\ast$ denote the class of
 symmetric subsets of $W$. Let $r > 0$, let $S_r = \set{u \in W : \norm{u} = r}$,
let $0 < b \le + \infty$, and let $\Gamma$ denote the group of odd homeomorphisms of $W$
that are the identity outside $I^{-1}(0,b)$. The pseudo-index of $M \in \A^\ast$ related to $i$, $S_r$,
and $\Gamma$ is defined \cite{MR84c:58014} by
\[
i^\ast(M) = \min_{\gamma \in \Gamma}\, i(\gamma(M) \cap S_r).
\]
The following critical point theorem generalizes \cite[Theorem 2.4]{MR713209}.
%and was recently proved in \cite{PerMeYY-I}.

\begin{theorem} \label{Theorem 2.4}
Let $A_0,\, B_0$ be symmetric subsets of $S_1$ such that $A_0$ is compact, $B_0$ is closed, and
\[
i(A_0) \ge k + m, \qquad i(S_1 \setminus B_0) \le k
\]
for some $k, m \in \N$. Assume that there exists $R > r$ such that
\[
\sup I(A) \le 0 < \inf I(B), \qquad \sup I(X) < b,
\]
where $A = \set{Ru : u \in A_0}$, $B = \set{ru : u \in B_0}$, and $X = \set{tu : u \in A,\, 0 \le t \le 1}$. For $j = k + 1,\dots,k + m$, let
\[
\A_j^\ast = \set{M \in \A^\ast : M \text{ is compact and } i^\ast(M) \ge j}
\]
and set
\[
c_j^\ast := \inf_{M \in \A_j^\ast}\, \max_{u \in M}\, I(u).
\]
Then
\[
\inf I(B) \le c_{k+1}^\ast \le \dotsb \le c_{k+m}^\ast \le \sup I(X),
\]
in particular, $0 < c_j^\ast < b$. If, in addition, $I$ satisfies the {\em \PS{c}}
 condition for all $c \in (0,b)$, then each $c_j^\ast$ is a critical value of $I$
and there are $m$ distinct pairs of associated critical points.
\end{theorem}

\begin{remark}
Constructions similar to the one in the proof of Theorem \ref{Theorem 2.4} have been used in Fadell and Rabinowitz \cite{MR57:17677} to prove bifurcation results for Hamiltonian systems, and in Perera and Szulkin \cite{PS} to obtain nontrivial solutions of $p$-Laplacian problems with nonlinearities that interact with the spectrum. See also \cite[Proposition 3.44]{MR2640827}.
\end{remark}

\subsection{Some estimates}

We have the following proposition from

\begin{proposition}[{\cite[Theorem 1.1, Lemma 2.1]{MM}}]
Let $1 < p < \infty$, $s \in (0,1)$, $N > sp$, $\alpha\in[0,sp)$, and let $S$ be as in \eqref{3}. Then
\begin{enumroman}
\item there exists a minimizer for $S$;
\item every minimizer $U$ is of  constant sign radially monotone; and if $\alpha>0$,
 then $U$ turns out to be radially  non-increasing around some point, which is zero.
\item for every minimizer $U$, there exists $\lambda_U > 0$ such that
\[
\int_{\R^{2N}} \frac{|U(x) - U(y)|^{p-2}\,(U(x) - U(y))\, (v(x) - v(y))}{|x - y|^{N+sp}}\, dx dy = \lambda_U \int_{\R^N}\frac{ |U|^{{p_s^\ast(\alpha)} - 2}}{|x|^\alpha}\,U\, v\, dx \quad \forall v \in {W}^{s,p}(\R^N).
\]
\end{enumroman}
\end{proposition}

\noindent
In the following, we shall fix a radially symmetric nonnegative decreasing minimizer $U = U(r)$
for $S$. Multiplying $U$ by a positive constant if necessary, we may assume that
\begin{equation} \label{7}
(- \Delta)_p^s\, U = U^{{p_s^\ast(\alpha)} - 1}.
\end{equation}
Testing this equation with $U$ and using \eqref{3} shows that
\begin{equation} \label{8}
\norm{U}^p = \pnorm[{p_s^\ast(\alpha)}]{U}^{{p_s^\ast(\alpha)}} = S^{\frac{N-\alpha}{sp-\alpha}}.
\end{equation}
For any $\eps > 0$, the function
\begin{equation} \label{9}
U_\eps(x) = \frac{1}{\eps^{(N-sp)/p}}\; U\bigg(\frac{|x|}{\eps}\bigg)
\end{equation}
is also a minimizer for $S$ satisfying \eqref{7} and \eqref{8}, so after a rescaling we may assume that $U(0) = 1$. Henceforth, $U$ will denote such a normalized (with respect to constant multiples and rescaling) minimizer and $U_\eps$ will denote the associated family of minimizers given by \eqref{9}. In the absence of an explicit formula for $U$, we will use the following asymptotic estimates.

\begin{lemma} \label{Lemma 1}
There exist constants $c_1, c_2 > 0$ and $\theta > 1$ such that for all $r \ge 1$,
\begin{equation} \label{10}
\frac{c_1}{r^{(N-sp)/(p-1)}} \le U(r) \le \frac{c_2}{r^{(N-sp)/(p-1)}}
\end{equation}
and
\begin{equation}\label{11}
\frac{U(\theta\, r)}{U(r)} \le \half.
\end{equation}
\end{lemma}

\begin{proof}
The inequalities in \eqref{10} were proved in Marano et al.\! \cite{MM}. They imply
\[
\frac{U(\theta\, r)}{U(r)} \le \frac{c_2}{c_1}\, \frac{1}{\theta^{(N-sp)/(p-1)}},
\]
and \eqref{11} follows for sufficiently large $\theta$.
\end{proof}

We now construct some auxiliary functions and estimate their norms. In what follows $\theta$ is the universal constant in Lemma \ref{Lemma 1} that depends only on $N$, $p$, and $s$. We may assume without loss of generality that $0 \in \Omega$.
For $\eps, \delta > 0$, let
\[
m_{\eps,\delta} = \frac{U_\eps(\delta)}{U_\eps(\delta) - U_\eps(\theta \delta)},
\]
let
\[
g_{\eps,\delta}(t) = \begin{cases}
0, & 0 \le t \le U_\eps(\theta \delta)\\[5pt]
m_{\eps,\delta}^p\, (t - U_\eps(\theta \delta)), & U_\eps(\theta \delta) \le t \le U_\eps(\delta)\\[5pt]
t + U_\eps(\delta)\, (m_{\eps,\delta}^{p-1} - 1), & t \ge U_\eps(\delta),
\end{cases}
\]
and let
\begin{equation}
\label{defG}
G_{\eps,\delta}(t) = \int_0^t g_{\eps,\delta}'(\tau)^{1/p}\, d\tau = \begin{cases}
0, & 0 \le t \le U_\eps(\theta \delta)\\[5pt]
m_{\eps,\delta}\, (t - U_\eps(\theta \delta)), & U_\eps(\theta \delta) \le t \le U_\eps(\delta)\\[5pt]
t, & t \ge U_\eps(\delta).
\end{cases}
\end{equation}
The functions $g_{\eps,\delta}$ and $G_{\eps,\delta}$ are nondecreasing and absolutely continuous. Consider the radially symmetric nonincreasing function
\[
u_{\eps,\delta}(r) = G_{\eps,\delta}(U_\eps(r)),
\]
which satisfies
\begin{equation} \label{20}
u_{\eps,\delta}(r) = \begin{cases}
U_\eps(r), & r \le \delta\\[5pt]
0, & r \ge \theta \delta.
\end{cases}
\end{equation}
We have the following estimates for $u_{\eps,\delta}$.

\begin{lemma} \label{Lemma 3}
There exists a constant $C = C(N,p,s) > 0$ such that for any $\eps \le \delta/2$,
\begin{gather}
\label{21} \norm{u_{\eps,\delta}}^p \le S^{\frac{N-\alpha}{sp-\alpha}} + C \left(\frac{\eps}{\delta}\right)^{(N-sp)/(p-1)},\\[10pt]
\label{22} \pnorm{u_{\eps,\delta}}^p \ge \begin{cases}
\dfrac{1}{C}\; \eps^{sp}\, \log \bigg(\dfrac{\delta}{\eps}\bigg), & N = sp^2\\[10pt]
\dfrac{1}{C}\; \eps^{sp}, & N > sp^2,
\end{cases}\\[12.5pt]
\label{23} \pnorm[{p_s^\ast(\alpha)}]{u_{\eps,\delta}}^{{p_s^\ast(\alpha)}} \ge S^{\frac{N-\alpha}{sp-\alpha}} - C \left(\frac{\eps}{\delta}\right)^{\frac{N-\alpha}{p-1}}.
\end{gather}
\end{lemma}

\begin{proof}
Using Brasco and Parini \cite[Lemma A.2]{BrPa} and testing the equation $(- \Delta)_p^s\, U_\eps = U_\eps^{{p_s^\ast(\alpha)} - 1}$ with $g_{\eps,\delta}(U_\eps) \in W^{s,p}_0(\Omega)$ gives
\begin{align*}
\norm{G_{\eps,\delta}(U_\eps)}^p & \le \int_{\R^{2N}} \frac{|U_\eps(x) - U_\eps(y)|^{p-2}\,(U_\eps(x) - U_\eps(y)) (g_{\eps,\delta}(U_\eps(x)) - g_{\eps,\delta}(U_\eps(y)))}{|x - y|^{N+sp}}\, dx dy  \\
&= \int_{\R^N} \frac{U_\eps(x)^{{p_s^\ast(\alpha)} - 1}\, g_{\eps,\delta}(U_\eps(x))}{|x|^\alpha}\, dx  \\
&= \pnorm[{p_s^\ast(\alpha)}]{U_\eps}^{{p_s^\ast(\alpha)}} + \int_{\R^N} \frac{(g_{\eps,\delta}(U_\eps(x)) - U_\eps(x))\, U_\eps(x)^{{p_s^\ast(\alpha)} - 1}}{|x|^\alpha}\, dx.
\end{align*}
We have $\pnorm[{p_s^\ast(\alpha)}]{U_\eps}^{{p_s^\ast(\alpha)}} = S^{\frac{N-\alpha}{sp-\alpha}}$ by \eqref{8},
\begin{align*}
g_{\eps,\delta}(U_\eps) - U_\eps \le U_\eps(\delta)\, m_{\eps,\delta}^{p-1} & = \frac{1}{\eps^{(N-sp)/p}}\; U\bigg(\frac{\delta}{\eps}\bigg)\! \left[1 - U\bigg(\dfrac{\theta \delta}{\eps}\bigg)\bigg/U\bigg(\dfrac{\delta}{\eps}\bigg)\right]^{-(p-1)}\\[5pt]
&\le 2^{p-1}\, c_2\, \frac{\eps^{(N-sp)/p(p-1)}}{\delta^{(N-sp)/(p-1)}}, %\,\,\quad \forall t \ge 0
\end{align*}
by \eqref{10} and \eqref{11},
\[
\int_{\R^N} \frac{U_\eps(x)^{{p_s^\ast(\alpha)} - 1}}{|x|^\alpha}\, dx = \eps^{(N-sp)/p} \int_{\R^N}\frac{ U(x)^{{p_s^\ast(\alpha)} - 1}}{|x|^\alpha}\, dx,
\]
and the last integral is finite by \eqref{10} again, so \eqref{21} follows.
Using \eqref{20},
\[
\int_{\R^N} u_{\eps,\delta}(x)^p\, dx \ge \int_{B_\delta(0)} u_{\eps,\delta}(x)^p\, dx = \int_{B_\delta(0)} U_\eps(x)^p\, dx = \eps^{sp} \int_{B_{\delta/\eps}(0)} U(x)^p\, dx,
\]
and the last integral is greater than or equal to
\[
\int_1^{\delta/\eps} U(r)^p\, r^{N-1}\, dr \ge c_1^p \int_1^{\delta/\eps} r^{-(N-sp^2)/(p-1)-1}\, dr
\]
by \eqref{10}. A direct evaluation of the integral on the right gives \eqref{22} since $\delta/\eps \ge 2$.
Using \eqref{20} again,
\begin{align*}
\int_{\R^N}\frac{ u_{\eps,\delta}(x)^{{p_s^\ast(\alpha)}}}{|x|^\alpha}\, dx &\ge \int_{B_\delta(0)}\frac{ u_{\eps,\delta}(x)^{{p_s^\ast(\alpha)}}}{|x|^\alpha}\, dx = \int_{B_\delta(0)} \frac{U_\eps(x)^{{p_s^\ast(\alpha)}}}{|x|^\alpha}\, dx \\
&= S^{\frac{N-\alpha}{sp-\alpha}} - \int_{B_{\delta/\eps}(0)^c}\frac{ U(x)^{{p_s^\ast(\alpha)}}}{|x|^\alpha}\, dx
\end{align*}
by \eqref{8}. By \eqref{10}, the last integral is less than or equal to
\[
c_2^{{p_s^\ast(\alpha)}} \int_{\delta/\eps}^\infty r^{(\alpha-N)/(p-1)-1}\, dr = \frac{(p - 1)\, c_2^{{p_s^\ast(\alpha)}}}{N-\alpha} \left(\frac{\eps}{\delta}\right)^{(N-\alpha)/(p-1)},
\]
so \eqref{23} follows.
\end{proof}

\noindent
We note that Lemma \ref{Lemma 3} gives the following estimate for
\[
S_{\eps,\delta}(\lambda) := \frac{\norm{u_{\eps,\delta}}^p - \lambda \pnorm{u_{\eps,\delta}}^p}{\pnorm[{p_s^\ast(\alpha)}]{u_{\eps,\delta}}^p}:
\]
there exists a constant $C = C(N,p,s) > 0$ such that for any $\eps \le \delta/2$,
\begin{equation} \label{40}
S_{\eps,\delta}(\lambda) \le \begin{cases}
S - \dfrac{\lambda}{C}\; \eps^{sp}\, \log \bigg(\dfrac{\delta}{\eps}\bigg) + C\, \bigg(\dfrac{\eps}{\delta}\bigg)^{sp}, & N = sp^2\\[10pt]
S - \dfrac{\lambda}{C}\; \eps^{sp} + C\, \bigg(\dfrac{\eps}{\delta}\bigg)^{(N-sp)/(p-1)}, & N > sp^2.
\end{cases}
\end{equation}

\begin{lemma}[{\cite[Proposition 3.1]{MPSY}}] \label{Proposition 2}
If $\lambda_k < \lambda_{k+1}$, then $\Psi^{\lambda_k}$ has a compact symmetric subset $E$ with $i(E) = k$.
\end{lemma}

In what follows
\[
\pi(u) = \frac{u}{\norm{u}},  \quad u \in W^{s,p}_0(\Omega) \setminus \set{0}
\]
are the radial projections onto
\[
\M = \bgset{u \in W^{s,p}_0(\Omega) : \norm{u} = 1}.
\]
.

\noindent
Now let $\theta$ be as in Lemma \ref{Lemma 1}, let $\eta \in C^\infty(\R^N,[0,1])$ be such that
\[
\eta(x) = \begin{cases}
0, & |x| \le 2 \theta\\[5pt]
1, & |x| \ge 3 \theta,
\end{cases}
\]
and let $\eta_\delta(x) = \eta\Big(\dfrac{x}{\delta}\Big)$ for $\delta > 0$.

\noindent
For $v \in E$, let $v_\delta = v \eta_\delta$, and let
\[
E_\delta = \set{\pi(v_\delta) : v \in E}.
\]

\begin{proposition}[{\cite[Proposition 3.2]{MPSY}}] \label{Proposition 3}
There exists a constant $C = C(N,\Omega,p,s,k) > 0$ such that for all sufficiently small $\delta > 0$,
\begin{gather}
\label{30} \frac{1}{C} \le \pnorm[q]{w} \le C \quad \forall w \in E_\delta,\, 1 \le q \le \infty,\\[7.5pt]
\label{31} \sup_{w \in E_\delta}\, \Psi(w) \le \lambda_k + C \delta^{N-sp},
\end{gather}
$E_\delta \cap \Psi_{\lambda_{k+1}} = \emptyset$, $i(E_\delta) = k$, and $\supp w \subset B_{2 \theta \delta}(0)^c$ for all $w \in E_\delta$. In particular, the supports of $w$ and $\pi(u_{\eps,\delta})$ are disjoint and hence $\pi(u_{\eps,\delta}) \not\in E_\delta$.
\end{proposition}

\subsection{Palais Smale condition}
\noindent

Weak solutions of problem \eqref{1} coincide with critical points of the $C^1$-functional
\begin{equation*} \label{6}
I_\lambda(u) = \frac{1}{p} \norm{u}^p - \frac{\lambda}{p} \pnorm{u}^p - \frac{1}{p_s^\ast(\alpha)}\pnorm[{p_s^\ast(\alpha)}]{u}^{{p_s^\ast(\alpha)}}
 \quad u \in W^{s,p}_0(\Omega).
\end{equation*}

Next we give the following compactness result, which
will be crucial for applying Theorem \ref{Theorem 2.1}  and Theorem \ref{Theorem 2.4} to our functional $I_\lambda$.

\begin{proposition}
	\label{Proposition 1}
Let $1 < p < \infty$, $s \in (0,1)$, $0<\alpha<sp<N$, and let $S$ be as in \eqref{3}. Then for any $\lambda \in \R$, $I_\lambda$ satisfies the {\em \PS{c}} condition for all $c < \dfrac{sp-\alpha}{p(N-\alpha)}\, S^{(N-\alpha)/(sp-\alpha)}$.
\end{proposition}

\begin{proof}

Let $c < \dfrac{sp-\alpha}{p(N-\alpha)}\, S^{(N-\alpha)/(sp-\alpha)}$ and let $\seq{u_j}$ be a sequence in $W^{s,p}_0(\Omega)$ such that
	\begin{gather}
	\label{88} I_\lambda(u_j) = \frac{1}{p} \norm{u_j}^p - \frac{\lambda}{p} \pnorm[p]{u_j}^p - \frac{1}{p_s^\ast(\alpha)} \pnorm[p_s^\ast(\alpha)]{u_j}^{p_s^\ast(\alpha)} = c + \o(1),\\[10pt]
	\label{99} \begin{split}
	I_\lambda'(u_j)\, v = & \int_{\R^{2N}} \frac{|u_j(x) - u_j(y)|^{p-2}\, (u_j(x) - u_j(y))\, (v(x) - v(y))}{|x - y|^{N+sp}}\, dx dy\\[5pt]
	& - \lambda \int_\Omega |u_j|^{p-2}\, u_j\, v\, dx - \int_\Omega\frac{|u_j|^{p_s^\ast(\alpha) - 2}}{|x|^\alpha}\, u_j\, v\, dx = \o(\norm{v}) \quad \forall v \in W^{s,p}_0(\Omega),
	\end{split}
	\end{gather}
	as $j \to \infty$. Then
	\[
	\frac{sp-\alpha}{p(N-\alpha)} \pnorm[p_s^\ast(\alpha)]{u_j}^{p_s^\ast(\alpha)} = I_\lambda(u_j) - \frac{1}{p}\, I_\lambda'(u_j)\, u_j = \o(\norm{u_j}) + O(1),
	\]
	which together with \eqref{88} and \eqref{V} shows that $\seq{u_j}$ is bounded in $W^{s,p}_0(\Omega)$. So a renamed subsequence of $\seq{u_j}$ converges to some $u$ weakly in $W^{s,p}_0(\Omega)$, strongly in $L^r(\Omega)$ for all $r \in [1,p_s^\ast)$, and a.e.\! in $\Omega$ (see Di Nezza et al.\! \cite[Corollary 7.2]{MR2944369}). Denoting by $p' = p/(p - 1)$ the H\"{o}lder conjugate of $p$, $|u_j(x) - u_j(y)|^{p-2}\, (u_j(x) - u_j(y))/|x - y|^{(N+sp)/p'}$ is bounded in $L^{p'}(\R^{2N})$ and converges to $|u(x) - u(y)|^{p-2}\, (u(x) - u(y))/|x - y|^{(N+sp)/p'}$ a.e.\! in $\R^{2N}$, and $(v(x) - v(y))/|x - y|^{(N+sp)/p} \in L^p(\R^{2N})$, so the first integral in \eqref{99} converges to
	\[
	\int_{\R^{2N}} \frac{|u(x) - u(y)|^{p-2}\, (u(x) - u(y))\, (v(x) - v(y))}{|x - y|^{N+sp}}\, dx dy
	\]
	for a further subsequence. Moreover,
	\[
	\int_\Omega |u_j|^{p-2}\, u_j\, v\, dx \to \int_\Omega |u|^{p-2}\, uv\, dx,
	\]
	and
	\[
	\int_\Omega\frac{|u_j|^{p_s^\ast(\alpha) - 2}}{|x|^\alpha}\, u_j\, v\, dx \to \int_\Omega \frac{|u|^{p_s^\ast(\alpha) - 2}}{|x|^\alpha}\, uv\, dx
	\]
	since

 $|u_j(x)|^{(p^\ast_s(\alpha)-2)}u_j(x)/|x |^{\alpha/p^\ast_s(\alpha)'}$ is bounded in $L^{p^\ast_s(\alpha)'}(\Omega)$ and converges to $|u(x)|^{p^\ast_s(\alpha)-2}u(x)/|x|^{\alpha/p^\ast_s(\alpha)'}$ a.e.\! in $\Omega$, and $v(x)/|x |^{\alpha/p^\ast_s(\alpha)} \in L^{p^\ast_s(\alpha)}(\Omega)$. So passing to the limit in \eqref{99} shows that $u \in W^{s,p}_0(\Omega)$ is a weak solution of \eqref{1}.
	
	Setting $\widetilde{u}_j = u_j - u$, we will show that $\widetilde{u}_j \to 0$ in $W^{s,p}_0(\Omega)$. We have
	\begin{equation} \label{111}
	\norm{\widetilde{u}_j}^p = \norm{u_j}^p - \norm{u}^p + \o(1)
	\end{equation}
	by Lemma \cite[Lemma 5]{PeSqYa2}, and
	\begin{equation}
	\pnorm[p_s^\ast(\alpha)]{\widetilde{u}_j}^{p_s^\ast(\alpha)} = \pnorm[p_s^\ast(\alpha)]{u_j}^{p_s^\ast(\alpha)} - \pnorm[p_s^\ast(\alpha)]{u}^{p_s^\ast(\alpha)} + \o(1)
	\end{equation}
	by the Br{\'e}zis-Lieb lemma \cite[Theorem 1]{MR699419}. Taking $v = u_j$ in \eqref{99} gives
	\begin{equation} \label{14}
	\norm{u_j}^p = \lambda \pnorm[p]{u}^p + \pnorm[p_s^\ast(\alpha)]{u_j}^{p_s^\ast(\alpha)} + \o(1)
	\end{equation}
	since $\seq{u_j}$ is bounded in $W^{s,p}_0(\Omega)$ and converges to $u$ in $L^p(\Omega)$, and testing \eqref{99} with $v = u$ gives
	\begin{equation} \label{12}
	\norm{u}^p = \lambda \pnorm[p]{u}^p + \pnorm[p_s^\ast(\alpha)]{u}^{p_s^\ast(\alpha)}.
	\end{equation}
	It follows from \eqref{111}--\eqref{12} and \eqref{3} that
	\[
	\norm{\widetilde{u}_j}^p = \pnorm[p_s^\ast(\alpha)]{\widetilde{u}_j}^{p_s^\ast(\alpha)} + \o(1) \le \frac{\norm{\widetilde{u}_j}^{p_s^\ast(\alpha)}}{S^{p_s^\ast(\alpha)/p}} + \o(1),
	\]
	so
	\begin{equation} \label{13}
	\norm{\widetilde{u}_j}^p \big(S^{p_s^\ast(\alpha)/p} - \norm{\widetilde{u}_j}^{p_s^\ast(\alpha) - p}\big) \le \o(1).
	\end{equation}
	On the other hand,
	\[
	\begin{aligned}
	c & = \frac{1}{p} \norm{u_j}^p - \frac{\lambda}{p} \pnorm[p]{u}^p - \frac{1}{p_s^\ast(\alpha)} \pnorm[p_s^\ast(\alpha)]{u_j}^{p_s^\ast(\alpha)} + \o(1) && \text{by \eqref{88}}\\[10pt]
	& = \frac{sp-\alpha}{p(N-\alpha)}\, \big(\norm{u_j}^p - \lambda \pnorm[p]{u}^p\big) + \o(1) && \text{by \eqref{14}}\\[10pt]
	& = \frac{sp-\alpha}{p(N-\alpha)}\, \big(\norm{\widetilde{u}_j}^p + \norm{u}^p - \lambda \pnorm[p]{u}^p\big) + \o(1) && \text{by \eqref{11}}\\[10pt]
	& = \frac{sp-\alpha}{p(N-\alpha)}\, \big(\norm{\widetilde{u}_j}^p + \pnorm[p_s^\ast(\alpha)]{u}^{p_s^\ast(\alpha)}\big) + \o(1) && \text{by \eqref{12}}\\[10pt]
	& \ge \frac{sp-\alpha}{p(N-\alpha)} \norm{\widetilde{u}_j}^p + \o(1),
	\end{aligned}
	\]
	so
	\begin{equation} \label{15}
	\limsup_{j \to \infty}\, \norm{\widetilde{u}_j}^p \le \frac{p(N-\alpha)c}{sp-\alpha} < S^{(N-\alpha)/(sp-\alpha)}.
	\end{equation}
	It follows from \eqref{13} and \eqref{15} that $\norm{\widetilde{u}_j} \to 0$.
\end{proof}

\medskip

\section{Proof of Theorem \ref{Theorem 1}}

In this section we prove Theorem \ref{Theorem 1}. For $0 < \lambda < \lambda_1$, mountain pass theorem and \eqref{40} will give us a positive critical level of $I_\lambda$ below the threshold level for compactness given in Proposition \ref{Proposition 1}. For $\lambda \ge \lambda_1$, we will use the abstract linking theorem, Theorem \ref{Theorem 2}.

\subsection{Case 1: $N \ge sp^2$ and $0 < \lambda < \lambda_1$}

We have
\[
I_\lambda(u) \ge \frac{1}{p} \left(1 - \frac{\lambda}{\lambda_1}\right) \norm{u}^p - \frac{1}{{p_s^\ast(\alpha)}\, S^{{p_s^\ast(\alpha)}/p}} \norm{u}^{{p_s^\ast(\alpha)}},
\]
so the origin is a strict local minimizer of $I_\lambda$. Fix $\delta > 0$ so small that $B_{\theta \delta}(0) \strictsubset \Omega$, so that $\supp u_{\eps,\delta} \subset \Omega$ by \eqref{20}. Noting that
\[
I_\lambda(Ru_{\eps,\delta}) = \frac{R^p}{p} \left(\norm{u_{\eps,\delta}}^p - \lambda \pnorm{u_{\eps,\delta}}^p\right) - \frac{R^{{p_s^\ast(\alpha)}}}{{p_s^\ast(\alpha)}} \pnorm[{p_s^\ast(\alpha)}]{u_{\eps,\delta}}^{{p_s^\ast(\alpha)}} \to - \infty \quad \text{as } R \to + \infty,
\]
fix $R_0 > 0$ so large that $I_\lambda(R_0 u_{\eps,\delta}) < 0$. Then let
\[
\Gamma = \set{\gamma \in C([0,1],W^{s,p}_0(\Omega)) : \gamma(0) = 0,\, \gamma(1) = R_0 u_{\eps,\delta}}
\]
and set
\[
c := \inf_{\gamma \in \Gamma}\, \max_{t \in [0,1]}\, I_\lambda(\gamma(t)) > 0.
\]
Since $t \mapsto tR_0 u_{\eps,\delta}$ is a path in $\Gamma$,
\begin{equation} \label{24}
c \le \max_{t \in [0,1]}\, I_\lambda(tR_0 u_{\eps,\delta}) = \frac{sp-\alpha}{p(N-\alpha)} \left(\frac{\norm{u_{\eps,\delta}}^p - \lambda \pnorm{u_{\eps,\delta}}^p}{\pnorm[{p_s^\ast(\alpha)}]{u_{\eps,\delta}}^p}\right)^{\frac{N-\alpha}{sp-\alpha}} = \frac{sp-\alpha}{p(N-\alpha)}\, S_{\eps,\delta}(\lambda)^{\frac{N-\alpha}{sp-\alpha}}.
\end{equation}
By \eqref{40},
\[
S_{\eps,\delta}(\lambda) \le \begin{cases}
S + \left(C - \dfrac{\lambda}{C} \abs{\log \eps}\right) \eps^{sp}, & N = sp^2\\[10pt]
S - \left(\dfrac{\lambda}{C} - C\, \eps^{(N-sp^2)/(p-1)}\right) \eps^{sp}, & N > sp^2,
\end{cases}
\]
so $S_{\eps,\delta}(\lambda) < S$ if $\eps > 0$ is sufficiently small. So
$$
c < \dfrac{sp-\alpha}{p(N-\alpha)}\, S^{\frac{N-\alpha}{sp-\alpha}}
$$
by \eqref{24}, and hence $I_\lambda$ satisfies the \PS{c} condition by Proposition \ref{Proposition 1}. Then $c$ is a critical level of $I_\lambda$ by the mountain pass theorem.

\subsection{Case 2: $N > sp^2$ and $\lambda > \lambda_1$ is not one of the eigenvalues $\lambda_k$}

We have $\lambda_k < \lambda < \lambda_{k+1}$ for some $k \in \N$, and then $i(\Psi^{\lambda_k}) = i(\M \setminus \Psi_{\lambda_{k+1}}) = k$ by \eqref{5}. Fix $\lambda'$ such that $\lambda_k < \lambda' < \lambda < \lambda_{k+1}$, and let $\delta > 0$ be so small that the conclusions of Proposition \ref{Proposition 3} hold with $\lambda_k + C \delta^{N-sp} < \lambda'$, in particular,
\begin{equation} \label{41}
\Psi(w) < \lambda' \quad \forall w \in E_\delta.
\end{equation}
Then take $A_0 = E_\delta$ and $B_0 = \Psi_{\lambda_{k+1}}$, and note that $A_0$ and $B_0$ are disjoint nonempty closed symmetric subsets of $\M$ such that
\[
i(A_0) = i(\M \setminus B_0) = k
\]
by Proposition \ref{Proposition 3} and \eqref{5}. Now let $0 < \eps \le \delta/2$, let $R > r > 0$, let $v_0 = \pi(u_{\eps,\delta}) \in \M \setminus E_\delta$, and let $A$, $B$ and $X$ be as in Theorem \ref{Theorem 2.1}.

For $u \in \Psi_{\lambda_{k+1}}$,
\[
I_\lambda(ru) \ge \frac{1}{p} \left(1 - \frac{\lambda}{\lambda_{k+1}}\right) r^p - \frac{1}{{p_s^\ast(\alpha)}\, S^{{p_s^\ast(\alpha)}/p}}\, r^{{p_s^\ast(\alpha)}}.
\]
Since $\lambda < \lambda_{k+1}$, it follows that $\inf I_\lambda(B) > 0$ if $r$ is sufficiently small.

Next we show that $I_\lambda \le 0$ on $A$ if $R$ is sufficiently large. For $w \in E_\delta$ and $t \ge 0$,
\[
I_\lambda(tw) \le \frac{t^p}{p} \left(1 - \frac{\lambda}{\Psi(w)}\right) \le 0
\]
by \eqref{41}. Now let $w \in E_\delta$ and $0 \le t \le 1$, and set $u = \pi((1 - t)\, w + tv_0)$. Clearly, $\norm{(1 - t)\, w + tv_0} \le 1$, and since the supports of $w$ and $v_0$ are disjoint by Proposition \ref{Proposition 3},
\[
\pnorm[{p_s^\ast(\alpha)}]{(1 - t)\, w + tv_0}^{{p_s^\ast(\alpha)}} = (1 - t)^{{p_s^\ast(\alpha)}} \pnorm[{p_s^\ast(\alpha)}]{w}^{{p_s^\ast(\alpha)}} + t^{{p_s^\ast(\alpha)}} \pnorm[{p_s^\ast(\alpha)}]{v_0}^{{p_s^\ast(\alpha)}}.
\]
In view of \eqref{30} and since
\begin{equation} \label{42}
\pnorm[{p_s^\ast(\alpha)}]{v_0}^{{p_s^\ast(\alpha)}} = \frac{\pnorm[{p_s^\ast(\alpha)}]{u_{\eps,\delta}}^{{p_s^\ast(\alpha)}}}{\norm{u_{\eps,\delta}}^{{p_s^\ast(\alpha)}}} \ge \frac{1}{S^{(N-\alpha)/(N-sp)}} + \O(\eps^{(N-sp)/(p-1)})
\end{equation}
by Lemma \ref{Lemma 3}, it follows that
\[
\pnorm[{p_s^\ast(\alpha)}]{u}^{{p_s^\ast(\alpha)}} = \frac{\pnorm[{p_s^\ast(\alpha)}]{(1 - t)\, w + tv_0}^{{p_s^\ast(\alpha)}}}{\norm{(1 - t)\, w + tv_0}^{{p_s^\ast(\alpha)}}} \ge \frac{1}{C}
\]
if $\eps$ is sufficiently small, where $C = C(N,\Omega,p,s,k) > 0$. Then
\[
I_\lambda(Ru) \le \frac{R^p}{p} - \frac{R^{{p_s^\ast(\alpha)}}}{{p_s^\ast(\alpha)}} \pnorm[{p_s^\ast(\alpha)}]{u}^{{p_s^\ast(\alpha)}} \le \frac{R^p}{p} - \frac{R^{{p_s^\ast(\alpha)}}}{{p_s^\ast(\alpha)}\, C} \le 0
\]
if $R$ is sufficiently large.
In view of \eqref{43} and Proposition \ref{Proposition 1}, it only remains to show that
$$
\sup I_\lambda(X) < \dfrac{sp-\alpha}{p(N-\alpha)}\, S^{\frac{N-\alpha}{sp-\alpha}},
$$
if $\eps$ is sufficiently small. Noting that
\[
X = \set{\rho\, \pi((1 - t)\, w + tv_0) : w \in E_\delta,\, 0 \le t \le 1,\, 0 \le \rho \le R},
\]
let $w \in E_\delta$ and $0 \le t \le 1$, and set $u = \pi((1 - t)\, w + tv_0)$. Then
\begin{align} \label{35}
\sup_{0 \le \rho \le R}\, I_\lambda(\rho u) &\le \sup_{\rho \ge 0}\, \left[\frac{\rho^p}{p} \left(1 - \lambda \pnorm{u}^p\right) - \frac{\rho^{{p_s^\ast(\alpha)}}}{{p_s^\ast(\alpha)}} \pnorm[{p_s^\ast(\alpha)}]{u}^{{p_s^\ast(\alpha)}}\right] = \frac{sp-\alpha}{p(N-\alpha)} \left[\frac{\left(1 - \lambda \pnorm{u}^p\right)^+}{\pnorm[{p_s^\ast(\alpha)}]{u}^p}\right]^{\frac{N-\alpha}{sp-\alpha}} \\
&= \frac{sp-\alpha}{p(N-\alpha)} \left[\frac{\left(\norm{(1 - t)\, w + tv_0}^p - \lambda \pnorm{(1 - t)\, w + tv_0}^p\right)^+}{\pnorm[{p_s^\ast(\alpha)}]{(1 - t)\, w + tv_0}^p}\right]^{\frac{N-\alpha}{sp-\alpha}}.   \notag
\end{align}

 From (3.17) in \cite[section 3.2]{MPSY},
\begin{equation}
\label{po}
\norm{(1 - t)\, w + tv_0}^p \le \frac{\lambda}{\lambda'}\, (1 - t)^p + t^p + C\, \eps^{N-(N-sp)\, q/p}.
\end{equation} where $q\in \ ]N(p - 1)/(N - sp), p[$.

On the other hand, since the supports of $w$ and $v_0$ are disjoint,
\begin{align} \label{49}
\pnorm{(1 - t)\, w + tv_0}^p &= (1 - t)^p \pnorm{w}^p + t^p \pnorm{v_0}^p,\\
\pnorm[{p_s^\ast(\alpha)}]{(1 - t)\, w + tv_0}^{{p_s^\ast(\alpha)}} &= (1 - t)^{{p_s^\ast(\alpha)}} \pnorm[{p_s^\ast(\alpha)}]{w}^{{p_s^\ast(\alpha)}} + t^{{p_s^\ast(\alpha)}} \pnorm[{p_s^\ast(\alpha)}]{v_0}^{{p_s^\ast(\alpha)}}.
\notag
\end{align}
By \eqref{41}, $\pnorm{w}^p = 1/\Psi(w) > 1/\lambda'$. By \eqref{30}, $\pnorm[{p_s^\ast(\alpha)}]{w}$ is bounded away from zero, and \eqref{42} implies that so is $\pnorm[{p_s^\ast(\alpha)}]{v_0}$ if $\eps$ is sufficiently small, so the last expression in \eqref{49} is bounded away from zero. It follows from \eqref{po} and \eqref{49} that
\[
\frac{\norm{(1 - t)\, w + tv_0}^p - \lambda \pnorm{(1 - t)\, w + tv_0}^p}{\pnorm[{p_s^\ast(\alpha)}]{(1 - t)\, w + tv_0}^p} \le \frac{1 - \lambda \pnorm{v_0}^p}{\pnorm[{p_s^\ast(\alpha)}]{v_0}^p} + C\, \eps^{N-(N-sp)\, q/p}.
\]
Since $v_0 = u_{\eps,\delta}/\norm{u_{\eps,\delta}}$, the right-hand side is less than or equal to
\[
S_{\eps,\delta}(\lambda) + C\, \eps^{N-(N-sp)\, q/p} \le S - \left(\frac{\lambda}{C} - C\, \eps^{(N-sp^2)/(p-1)} - C\, \eps^{(N-sp)(1-q/p)}\right) \eps^{sp}
\]
by \eqref{40}. Since $N > sp^2$ and $q < p$, it follows from this that the last expression in \eqref{35} is strictly less than $\dfrac{sp-\alpha}{p(N-\alpha)}\, S^{\frac{N-\alpha}{sp-\alpha}}$ if $\eps$ is sufficiently small.

\subsection{Case 3: $[(N-\alpha)N+\alpha\,s\,p\,(1+p)]/(N + s) > sp^2$, and $\lambda = \lambda_k$}

Let $\lambda = \lambda_k < \lambda_{k+1}$, let $\delta > 0$ be so small that the conclusions of Proposition \ref{Proposition 3} hold with $\lambda_k + C \delta^{N-sp} < \lambda_{k+1}$, in particular, $\Psi(w) < \lambda_{k+1}$ for all $w \in E_\delta$, and take $A_0 = E_\delta$ and $B_0 = \Psi_{\lambda_{k+1}}$ as in the last subsection. Then let $0 < \eps \le \delta/2$, let $R > r > 0$, let $v_0 = \pi(u_{\eps,\delta}) \in \M \setminus E_\delta$, and let $A$, $B$ and $X$ be as in Theorem \ref{Theorem 2.1}. As before, $\inf I_\lambda(B) > 0$ if $r$ is sufficiently small and
\[
I_\lambda(R\, \pi((1 - t)\, w + tv_0)) \le 0 \quad \forall w \in E_\delta,\, 0 \le t \le 1
\]
if $R$ is sufficiently large. On the other hand,
\[
I_\lambda(tw) \le \frac{t^p}{p} \left(1 - \frac{\lambda_k}{\Psi(w)}\right) \le C R^p \delta^{N-sp} \quad \forall w \in E_\delta,\, 0 \le t \le R
\]
by \eqref{31}, where $C$ denotes a generic positive constant independent of $\eps$ and $\delta$. It follows that
\[
\sup I_\lambda(A) \le C R^p \delta^{N-sp}< \inf I_\lambda(B)
\]
if $\delta$ is sufficiently small. As in the last proof, it only remains to show that (see \eqref{35})
\begin{equation} \label{48}
\sup_{(w,t) \in E_\delta \times [0,1]}\, \frac{\norm{(1 - t)\, w + tv_0}^p - \lambda_k \pnorm{(1 - t)\, w + tv_0}^p}{\pnorm[{p_s^\ast(\alpha)}]{(1 - t)\, w + tv_0}^p} < S
\end{equation}
if $\eps$ and $\delta$ are suitably small.\

Now let $\delta = \eps^\mu$ with $\mu \in (0,1)$, %let $J_q := \int_{A_3} \frac{|w(x)|^{p-q}\, v_0(y)^q}{|x - y|^{N+sp}}\, dx dy$,
 from (3.22) in\cite[Section 3.3]{MPSY}
\begin{equation} \label{46}
\norm{(1 - t)\, w + tv_0}^p \le (1 - t)^p + t^p + \widetilde{J}_1 + \widetilde{J}_{p-1},
\end{equation}
where
\[
\widetilde{J}_q  \le C\, (1 - t)^{p-q}\, \eps^{(N-sp)[p\, (p-q-1)\, \mu + q]/p\, (p-1)}.
\]
Young's inequality then gives
\begin{equation} \label{52}
\widetilde{J}_q \le \frac{\kappa}{3}\, (1 - t)^{{p_s^\ast(\alpha)}} + C\, \eps^{sp + \beta_q(\mu)} \kappa^{- \gamma_q}
\end{equation}
for any $\kappa > 0$, where
\begin{equation*}
\beta_q(\mu) = \frac{[N(N-\alpha)-(N+s)sp^2+\alpha\,s\,p\,(p+1)](p - 1)(p - q) -( N-\alpha)p\, (N - sp)(p - q - 1)(\mu_0 - \mu)}{(p - 1)[(N - sp)\, q + p(sp-\alpha)]},
\end{equation*}
and
\begin{equation*}
\mu_0 = \frac{N - sp^2}{N - sp}, \qquad \gamma_q = \frac{(N - sp)(p - q)}{(N-\alpha)p - (N - sp)(p - q)}.
\end{equation*}
Then
\begin{equation} \label{47}
\norm{(1 - t)\, w + tv_0}^p \le (1 - t)^p + t^p + \frac{2 \kappa}{3}\, (1 - t)^{{p_s^\ast(\alpha)}} + C\, \eps^{sp} \left(\eps^{\beta_1(\mu)} \kappa^{- \gamma_1} + \eps^{\beta_{p-1}(\mu)} \kappa^{- \gamma_{p-1}}\right)
\end{equation}
by \eqref{46} and \eqref{52}.
Using $[(N-\alpha)N+\alpha sp(1+p)]/(N + s) > sp^2$, we fix $\mu < \mu_0$ so close to $\mu_0$ that $\beta_q(\mu) > 0$ for $q = 0, 1, p - 1, p$. By \eqref{31} and Young's inequality,
\begin{equation} \label{53}
\lambda_k\, (1 - t)^p \pnorm{w}^p \ge (1 - t)^p \left(1 - C\, \eps^{(N-sp)\, \mu}\right) \ge (1 - t)^p - \frac{\kappa}{3}\, (1 - t)^{{p_s^\ast(\alpha)}} - C\, \eps^{sp + \beta_0(\mu)} \kappa^{- \gamma_0}.
\end{equation}
By \eqref{47}, \eqref{49}, and \eqref{53}, the quotient $Q(w,t)$ in \eqref{48} satisfies
\begin{equation} \label{54}
Q(w,t) \le \frac{\big(1 - \lambda_k \pnorm{v_0}^p\big)\, t^p + \kappa\, (1 - t)^{{p_s^\ast(\alpha)}} + C\, \eps^{sp + \beta(\mu)} \kappa^{- \gamma}}{\left[(1 - t)^{{p_s^\ast(\alpha)}} \pnorm[{p_s^\ast(\alpha)}]{w}^{{p_s^\ast(\alpha)}} + t^{{p_s^\ast(\alpha)}} \pnorm[{p_s^\ast(\alpha)}]{v_0}^{{p_s^\ast(\alpha)}}\right]^{p/{p_s^\ast(\alpha)}}},
\end{equation}
where
\[
\beta(\mu) = \min \set{\beta_0(\mu),\beta_1(\mu),\beta_{p-1}(\mu)} > 0, \qquad \gamma = \max \set{\gamma_0,\gamma_1,\gamma_{p-1}} = \frac{N-sp}{sp-\alpha} .
\]
As before, the denominator is bounded away from zero if $\eps$ is sufficiently small, so it follows that
\[
\sup_{(w,t) \in E_{\eps^\mu} \times [0,t_0)}\, Q(w,t) \leq C(t_0^p+\kappa+\eps^{sp+\beta(\mu)}\kappa^{-\gamma})< S
\]
for some $t_0 > 0$ if $\kappa$ and $\eps$ are sufficiently small. For $t \ge t_0$, rewriting the right-hand side of \eqref{54} as
\[
\frac{\dfrac{1 - \lambda_k \pnorm{v_0}^p}{\pnorm[{p_s^\ast(\alpha)}]{v_0}^p} + \dfrac{\kappa\, (1 - t)^{{p_s^\ast(\alpha)}} + C\, \eps^{sp + \beta(\mu)} \kappa^{- \gamma}}{t^p \pnorm[{p_s^\ast(\alpha)}]{v_0}^p}}{\left[\dfrac{\pnorm[{p_s^\ast(\alpha)}]{w}^{{p_s^\ast(\alpha)}}}{t^{{p_s^\ast(\alpha)}} \pnorm[{p_s^\ast(\alpha)}]{v_0}^{{p_s^\ast(\alpha)}}}\, (1 - t)^{{p_s^\ast(\alpha)}} + 1\right]^{p/{p_s^\ast(\alpha)}}}
\]
gives $Q(w,t) \le g((1 - t)^{{p_s^\ast(\alpha)}})$, where
\[
g(\tau) = \frac{S_{\eps,\eps^\mu}(\lambda_k) + C \left(\kappa \tau + \eps^{sp + \beta(\mu)} \kappa^{- \gamma}\right)}{(1 + C^{-1}\, \tau)^{p/{p_s^\ast(\alpha)}}}, \qquad C=C(N, p, s, t_0).
\]
Since $0 \le (1 - t)^{{p_s^\ast(\alpha)}} < 1$, then
\[
Q(w,t) \le S_{\eps,\eps^\mu}(\lambda_k) + C \big(\kappa + \eps^{sp + \beta(\mu)} \kappa^{- \gamma}\big).
\]
If $S_{\eps_j,\eps_j^\mu}(\lambda_k) < S/2$ for some sequence $\eps_j \to 0$, then the right-hand side is less than $S$ for sufficiently small $\kappa$ and $\eps = \eps_j$ with sufficiently large $j$, so we may assume that $S_{\eps,\eps^\mu}(\lambda_k) \ge S/2$ for all sufficiently small $\eps$. Then it is easily seen that if $\kappa \le (p/{p_s^\ast(\alpha)})\, S/2\, C(C + 1)$, then $g'(\tau) \le 0$ for all $\tau \in [0,1]$ and hence the maximum of $g((1 - t)^{{p_s^\ast(\alpha)}})$ on $[t_0,1]$ occurs at $t = 1$. So, we reach
\[
Q(w,t) \le S_{\eps,\eps^\mu}(\lambda_k) + C\, \eps^{sp + \beta(\mu)} \kappa^{- \gamma} \le S - \left(\frac{\lambda_k}{C} - C\, \eps^{\beta_p(\mu)} - C\, \eps^{\beta(\mu)} \kappa^{- \gamma}\right) \eps^{sp}
\]
by \eqref{40}, and the desired conclusion follows for sufficiently small $\kappa$ and $\eps$.

\subsection{Case 4: $[N^2(N-\alpha) + s^3 p^3 +\alpha s p (N-sp)]/N\, (N + s -\alpha) > sp^2,\,
\bdry{\Omega} \in C^{1,1}$, and $\lambda = \lambda_k$}

\noindent
From the arguments in\cite[Section 3.4]{MPSY}, \eqref{31} can now be strengthened to
\begin{equation} \label{64}
\sup_{w \in E_\delta}\, \Psi(w) \le \lambda_k + C \delta^N.
\end{equation}
Proceeding as in the last subsection, we have to verify \eqref{48} for suitably small $\eps$ and $\delta$.
Since the argument is similar, we only point out where it differs. \\

From the arguments in \cite[Section 3.4]{MPSY}

\[
\widetilde{J}_q \le C\,(1-t)^{p-q} \eps^{\{p\, [(p-q-1)\, N + sq]\, \mu + (N-sp)\, q\}/p\, (p-1)}.
\]
Then \eqref{52} holds with
\begin{align*}
\beta_q(\mu)&= \frac{[N^2\,(N-\alpha) + s^3 p^3 - N sp^2\, (N + s - \alpha)+\alpha\,s\,p\,(N-sp)](p - 1)(p - q)}{(N - sp)[(N - sp)\, q +p( sp-\alpha)](p - 1)}\\
   &- \frac{(N-\alpha)p\, (N - sp)[N\, (p - q - 1) + sq](\mu_0 - \mu)}{(N - sp)[(N - sp)\, q + p(sp-\alpha)](p - 1)},
\end{align*}
and so does \eqref{53} by \eqref{64}. Using
$$
[N^2(N-\alpha) + s^3 p^3 +\alpha s p (N-sp)]/N\, (N + s -\alpha) > sp^2,
$$
we fix $\mu < \mu_0$ so close to $\mu_0$ that $\beta_q(\mu) > 0$ for $q = 0, 1, p - 1, p$ and proceed as before.

\def\cdprime{$''$}

\section{Proof of Theorem \ref{Theorem 2}}
\noindent
 By Proposition \ref{Proposition 1}, $I_\lambda$ satisfies the \PS{c} condition for all $c < \dfrac{sp-\alpha}{p(N-\alpha)}\, S^{(N-\alpha)/(sp-\alpha)}$, so we apply Theorem \ref{Theorem 2.4} with $b = \dfrac{sp-\alpha}{p(N-\alpha)}\, S^{(N-\alpha)/(sp-\alpha)}$. By Lemma \ref{Proposition 2}, $\Psi^{\lambda_{k+m}}$ has a compact symmetric subset $A_0$ with
	\[
	i(A_0) = k + m.
	\]
	We take $B_0 = \Psi_{\lambda_{k+1}}$, so that
	\[
	i(S_1 \setminus B_0) = k
	\]
	by \eqref{5}. Let $R > r > 0$ and let $A$, $B$ and $X$ be as in Theorem \ref{Theorem 2.4}. For $u \in B_0$,
	\[
	I_\lambda(ru) \ge \frac{r^p}{p} \left(1 - \frac{\lambda}{\lambda_{k+1}}\right) - \frac{r^{p_s^\ast(\alpha)}}{p_s^\ast(\alpha)\, S^{p_s^\ast(\alpha)/p}}
	\]
	by \eqref{3}. Since $\lambda < \lambda_{k+1}$ and $p_s^\ast(\alpha) > p$, it follows that $\inf I_\lambda(B) > 0$ if $r$ is sufficiently small. For $u \in A_0 \subset \Psi^{\lambda_{k+1}}$,
	\[
	I_\lambda(Ru) \le \frac{R^p}{p} \left(1 - \frac{\lambda}{\lambda_{k+1}}\right) - \frac{R^{p_s^\ast(\alpha)}}{p_s^\ast(\alpha) V_\alpha(\Omega)^{(s p-\alpha))/(N-sp)} \lambda_{k+1}^{p_s^\ast(\alpha)/p}}
	\]
	by \eqref{V} and the H\"{o}lder inequality, so there exists $R > r$ such that $I_\lambda \le 0$ on $A$. For $u \in X$,
	\begin{align*}
	I_\lambda(u) & \le \frac{\lambda_{k+1} - \lambda}{p} \int_\Omega |u|^p\, dx - \frac{1}{p_s^\ast(\alpha) V_{\alpha}(\Omega)^{(s p-\alpha)/(N-sp)}} \left(\int_\Omega |u|^p\, dx\right)^{p_s^\ast(\alpha)/p}\\[10pt]
	& \le \sup_{\rho \ge 0}\, \left[\frac{(\lambda_{k+1} - \lambda)\, \rho}{p} - \frac{\rho^{p_s^\ast(\alpha)/p}}{p_s^\ast(\alpha) V_{\alpha}(\Omega)^{(s p-\alpha)/(N-sp)}}\right]\\[10pt]
	& = \frac{sp-\alpha}{p(N-\alpha)}\, V_{\alpha}(\Omega) (\lambda_{k+1} - \lambda)^{(N-\alpha)/(sp-\alpha)}.
	\end{align*}
	So
	\[
	\sup I_\lambda(X) \le \frac{sp-\alpha}{p(N-\alpha)}\, V_{\alpha}(\Omega) (\lambda_{k+1} - \lambda)^{(N-\alpha)/(sp-\alpha)} < \dfrac{sp-\alpha}{p(N-\alpha)}\, S^{(N-\alpha)/(sp-\alpha)}
	\]
	by \eqref{555}. Theorem \ref{Theorem 2.4} now gives $m$ distinct pairs of (nontrivial) critical points $\pm\, u^\lambda_j,\, j = 1,\dots,m$ of $I_\lambda$ such that
	\begin{equation*}
	0 < I_\lambda(u^\lambda_j) \le \frac{sp-\alpha}{p(N-\alpha)}\, V_{\alpha}(\Omega) (\lambda_{k+1} - \lambda)^{(N-\alpha)/(sp-\alpha)} \to 0 \text{ as } \lambda \nearrow \lambda_{k+1}.
	\end{equation*}
	Then
	\[
	|u^\lambda_j|_{p_s^\ast(\alpha)}^{p_s^\ast(\alpha)} = \frac{(N-\alpha)p}{sp-\alpha} \left[I_\lambda(u^\lambda_j) - \frac{1}{p}\, I_\lambda'(u^\lambda_j)\, u^\lambda_j\right] = \frac{(N-\alpha)p}{sp-\alpha}\, I_\lambda(u^\lambda_j) \to 0
	\]
	and hence $u^\lambda_j \to 0$ in $L^p(\Omega)$ also by \eqref{V}, so
	\[
	\|u^\lambda_j\|^p = p\, I_\lambda(u^\lambda_j) + \lambda\, |u^\lambda_j|_p^p + \frac{p}{p_s^\ast(\alpha)}\, |u^\lambda_j|_{p_s^\ast(\alpha)}^{p_s^\ast(\alpha)} \to 0.
	\]

\end{document}